\documentclass{amsart}
\usepackage{amsfonts,amssymb,amsmath,amsthm}
\usepackage{url}
\usepackage{enumerate}

\newtheorem{thrm}{Theorem}[section]
\newtheorem{lem}[thrm]{Lemma}

\newtheorem{cor}[thrm]{Corollary}
\theoremstyle{definition}

\newtheorem{remark}[thrm]{Remark}
\numberwithin{equation}{section}

\newcommand{\C}{\mathbb{C}}
\newcommand{\R}{\mathbb{R}}
\newcommand{\Z}{\mathbb{Z}}

\newcommand{\TT}{\mathbb{T}}%Torus
\newcommand{\m}{\mathrm{m}}

\author{Zahraa Issa, Matilde Lal\'{i}n}

 \address{Zahraa Issa: D\'epartement de math\'ematiques et de statistique,
                                    Universit\'e de Montr\'eal.
                                    CP 6128, succ. Centre-ville.
                                     Montreal, QC H3C 3J7, Canada} \email{issaz@dms.umontreal.ca}

\address{Matilde Lal\'in:  D\'epartement de math\'ematiques et de statistique,
                                    Universit\'e de Montr\'eal.
                                    CP 6128, succ. Centre-ville.
                                     Montreal, QC H3C 3J7, Canada}\email{mlalin@dms.umontreal.ca}
\thanks{This work was supported by NSERC Discovery Grant 355412-2008, FQRNT Subvention \'etablissement
de nouveaux chercheurs 144987, and a start-up grant from the Universit\'e de Montr\'eal.}

\keywords{Mahler measure, polynomial}

\subjclass[2010]{Primary 11R06; Secondary 11R09}

\begin{document}
\title{A generalization of a theorem of Boyd and Lawton}

\begin{abstract} The Mahler measure of a nonzero $n$-variable polynomial $P$ is the integral of
$\log|P|$ on the unit $n$-torus. A result of Boyd and Lawton says that the Mahler measure of a multivariate polynomial is the limit of Mahler measures of univariate polynomials. We prove the analogous
result for different extensions of Mahler measure such as generalized Mahler measure (integrating the maximum of $\log|P|$ for possibly different $P$'s), 
multiple Mahler measure (involving products of $\log|P|$ for possibly different $P$'s), and higher Mahler measure (involving $\log^k|P|$).
 
\end{abstract}
\maketitle

\section{Introduction}

The Mahler measure of a nonzero polynomial $P(x_1,\dots,x_n) \in \C[x_1,\dots,x_n]$ is defined by 
\[
\m(P) :=\frac{1}{(2\pi i)^n}\int_{\TT^n} 
\log \left|P(x_1,\dots,x_n)\right| \frac{dx_1}{x_1}\cdots \frac{dx_n}{x_n},\]
where $\TT^n=\{ (z_1,\dots,z_n) \in \C^n : |z_1|=\dots=|z_n|\}$ is the unit torus in dimention $n$. 
This formula has a particularly simple expression for univariate polynomials. If $P(x) = a \prod_i(x-\alpha_i)$, then Jensen's formula implies that 
$\m(P)= \log |a| + \sum_i \max \{0,\log|\alpha_i|\}$. In fact, Lehmer \cite{Le} considered first the measure for univariate polynomials which was later extended 
to multivariate polynomials by Mahler \cite{Mah}. Lehmer's motivation for considering this object was a method to construct large prime numbers that generalizes Mersenne's sequence.
Mahler, on the other hand, was interested in relating heights of products of polynomials with the heights of the factors. The Mahler measure is a height which is multiplicative, and therefore it was a natural object for Mahler to consider. 

Boyd and Lawton  proved the following useful and interesting result.
\begin{thrm}\cite{B, B2, L}\label{Thm}
Let $P(x_1,\dots,x_n) \in \C[x_1,\dots,x_n]$ and ${\bf r} =
(r_1,\dots,r_n),\,r_i \in \Z_{>0}.$  Define $P_{{\bf r}}(x)$ as
\[P_{{\bf r}}(x) = P(x^{r_1},\dots,x^{r_n}),\] and let
\[q({\bf r}) = \min \left \{H({\bf t}): \,{\bf t} = (t_1,\dots,t_n) \in \Z^n,\,
{\bf t} \neq (0,\dots,0), \sum_{j=1}^nt_jr_j = 0 \right\},\]
where $H({\bf t}) = \max\{|t_j|:\,1\leq j \leq n\}$.  Then
\[\lim_{q({\bf r})\to \infty}\m(P_{{\bf r}}) = \m(P).\]
\end{thrm} 
This result implies that the multivariate Mahler measure is a limit of univariate Mahler measures. In particular, it gives evidence that the extension to multivariate polynomials is the right generalization. 

The Mahler measure of multivariate polynomials often yields special values of the Riemann zeta function and $L$-functions, thus one can construct sequences of numbers that approach these special values in this way. 

In addition, this theorem has consequences in terms of limit points of Mahler measure. The most famous open question in this area is the so called Lehmer's question.
{\em Is there a constant $c > 0$ such that for every polynomial $P \in \Z[x]$ with $m(P) >0$, then $m(P)\geq c$?} Thus, Theorem \ref{Thm} tells us that given a multivariate polynomial whose measure is smaller than a certain constant $c$, we can generate infinitely many univariate polynomials with the same
property. 

In this work, we are going to consider two extensions of Mahler measure. 

Given $P_1,\dots, P_s \in \C[x_1, \dots, x_n],$ (not necessarily distinct) nonzero polynomials, the {\em generalized Mahler measure} is defined in \cite{GO} by 
\[\m_{\max}(P_1,\dots,P_s):=\frac{1}{(2\pi i)^n}\int_{\TT^n}\max\{\log|P_1(x_1,\dots,x_n)|,\dots
, \log|P_s(x_1,\dots,x_n)|\} \frac{dx_1}{x_1}\cdots \frac{dx_n}{x_n}.\]
On the other hand, the {\em multiple  Mahler measure} is defined in \cite{KLO} by
\[\m(P_1,\dots, P_s) :=\frac{1}{(2\pi i)^n}\int_{\TT^n}\log|P_1(x_1, \dots, x_n)|\cdots
\log|P_s(x_1, \dots, x_n)| \frac{dx_1}{x_1}\cdots \frac{dx_n}{x_n}.\]
For the particular case in which $P_1=\dots=P_s=P$, the multiple Mahler meausure is called {\em higher Mahler measure}
\[\m_s(P) :=\frac{1}{(2\pi i)^n}\int_{\TT^n}\log^s|P(x_1, \dots, x_n)|\frac{dx_1}{x_1}\cdots \frac{dx_n}{x_n}.\]

These objects have been related to special values of the Riemann zeta function and $L$-functions (\cite{GO,L08} for generalized Mahler measure, \cite{KLO, Sa10, Sa, BS, BBSW} for multiple Mahler measure), but the nature of this relationship is less well understood than in the classical case. 

Our goal in this note is to prove the equivalent for Theorem \ref{Thm} for these generalizations.

\begin{thrm}\label{superBL} Let $P_1,\dots,P_s \in \C[x_1,\dots,x_n],$  and ${\bf
r}$ as before. Then
\begin{enumerate}
\item \[\lim_{q({\bf r})\to \infty}\m_{\max}({P_1}_{{\bf r}},\dots, {P_s}_{{\bf r}}) =
\m_{\max}(P_1,\dots,P_s).\]
\item  \[\lim_{q({\bf r})\to \infty}\m({P_1}_{{\bf r}},\dots, {P_s}_{{\bf r}}) =
\m(P_1,\dots,P_s).\]
\item If $P_1=\dots=P_s=P$, 
\[\lim_{q({\bf r})\to \infty}\m_s({P}_{{\bf r}})=\m_s(P).\] 
 \end{enumerate}
\end{thrm}

\section{Some preliminary results}
The difficulty in obtaining Theorem \ref{superBL} lies in the case where (some of) the polynomials vanish in the domain of integration and the logarithm is not bounded. This problem already appears 
in the proof of Theorem \ref{Thm}. The key result for solving this is a theorem by Lawton \cite{L}. 

Let $\mu_n$ denote the Lebesgue measure in the torus $\TT^n$. 
\begin{thrm}[\cite{L}, Theorem 1] \label{thm1L} Let $P(x)\in \C[x]$ be a monic polynomial  and let $k=$ number of nonzero coefficients of $P$. Then if $k\geq 2$, there is a positive
constant $C_k$ that depends only on $k$ such that 
\[\mu_1(\{z \in \TT \, :\, |P(z)|<y\})\leq C_ky^\frac{1}{k-1},\]
for any real number $y>0$. 
\end{thrm}
The strength of this result lies in the fact that the constant is absolute and depends on the number of nonzero coefficients of $P$ but it does not depend on $P$. 

Notice that we can always assume that the polynomials involved in multiple Mahler measure have at least two nonzero monomials, 
since $\log|ax^k|$ is a constant and can be easily extracted from the integral. It should be noted that the above theorem remains true for $k=1$ if $y$ is sufficiently small (i.e., $y<|a|$) and $C_1=0$.

It is not hard to prove a result where the constant depends on $P$. For example,
\begin{lem}[\cite{EW}, Lemma 3.8, pg. 58] Let $P(x_1,\dots,x_n) \in  \C[x_1,\dots,x_n]$. There there are constants $C_P,\delta_P$ that depend on $P$ such that 
\begin{equation}\label{multi} 
\mu_n(\{(z_1,\dots,z_n) \in \TT^n \, :\, |P(z_1,\dots,z_n)|<y\})\leq C_Py^{\delta_P},
\end{equation}
for small $y>0$.  
\end{lem}

In what follows, we will denote by
\[S_n(P,y)=\{(z_1,\dots,z_n) \in \TT^n \, :\, |P(z_1,\dots,z_n)|<y\},\]
where the $n$ depends on the number of variables involved. Thus $n \geq$ number of variables of $P$. We will write $S(P,y)$ for $S_1(P,y)$.

The following elementary lemma will be useful to bound integrals.

\begin{lem}\label{integral} Let $\ell$ be a  positive integer and $y,\delta>0$. Then
\begin{eqnarray*}
J_{\ell,\delta}(y)&:=& (-1)^\ell \int_0^y \log^\ell z d \left(z^\delta \right)\\
%\end{eqnarray*}
%\begin{eqnarray*}
&=&y^\delta \left((-1)^\ell \log^\ell y +\frac{\ell}{\delta} (-1)^{\ell-1} \log^{\ell-1} y + \frac{\ell (\ell-1)}{\delta^2} (-1)^{\ell-2} \log^{\ell-2} y+\dots\right.\\
&&\left.+ 
\frac{ \ell (\ell-1)\cdots 2}{\delta^{\ell-1}} (-1)\log y + \frac{\ell!}{\delta^\ell}\right).
\end{eqnarray*}
\end{lem}
\begin{proof}
The proof is easily obtained by repeated integration by parts. 
\end{proof}
%\begin{proof}We proceed by induction on $\ell$. For $\ell=1$, we have, by integration by parts,
%\[-\int_0^y \log z d \left(z^\delta \right)= -y^\delta \log y + \int_0^y z^{\delta-1} dz=  y^\delta\left(-\log y +\frac{1}{\delta}\right).\]
%Thus, assume that the case $\ell-1$ is true. For $\ell$ we have, by integration by parts,
%\[(-1)^\ell\int_0^y \log^{\ell} z d \left(z^\delta \right)= y^\delta (-1)^\ell\log^{\ell} y -(-1)^\ell \ell \int_0^y \log^{\ell-1}z  z^{\delta -1} dz. \]
%By induction hypothesis, the previous formula equals
%\[y^\delta (-1)^\ell \log^{\ell} y - \frac{\ell}{\delta }y^\delta  \left((-1)^{\ell-1} \log^{\ell-1} y +\frac{(\ell-1)}{\delta } (-1)^{\ell-2}\log^{\ell-2} y +\dots\right.\]
%\[\left.+\frac{ (\ell-1)\cdots 2}{\delta^{\ell-2}} (-1)\log y + \frac{(\ell-1)!}{\delta^{\ell-1}}\right), \]
%which gives the desired result by induction. 
% \end{proof}
\begin{cor}\label{limI} For $0<y\leq 1$ we have
\[0\leq J_{\ell, \delta}(y) \leq y^\delta(\ell+1)! \max \left \{\frac{1}{\delta},(-\log y)\right\}^\ell.\]
In other words, 
\[ \lim_{y \rightarrow 0} J_{\ell, \delta}(y)=0.\]
\end{cor}
For the remainder of this work, we will denote by
\begin{equation}\label{I}
I_{\ell,k}(y):= J_{\ell, \frac{1}{k-1}}(y)= (-1)^\ell \int_0^y \log^\ell z d \left(z^\frac{1}{k-1} \right).
\end{equation}
We finish this section by recalling the statement of the following extension of H\"older's inequality:
\begin{lem}\label{superHolder}
Let $S$ a measurable set of $\R^n$ or $\C^n$ and $f_1,\dots,f_s$ measurable complex or real valued functions. Then 
\[\int_S |f_1\cdots f_s| dx \leq \left(\int_S |f_1|^s dx\right)^\frac{1}{s} \cdots \left(\int_S |f_s|^s dx\right)^\frac{1}{s}.\]
\end{lem}

\section{Integration over combinations of $S(P,y)$}

In this section, we consider the integration over sets resulting from combining the different $S(P,y)$'s.

\begin{lem} \label{Lemma}
 Let $P(x) \in \C[x]$ a polynomial having $k\geq 2$ non-zero complex coefficients each having modulus $\geq 1$. Let $0<y\leq1$.
Then \[0 \leq  (-1)^\ell \int_{S(P,y)} \log^\ell|P(x)| \frac{dx }{x} \leq  C_k I_{\ell,k}(y).\]
Analogously, if  $P(x_1,\dots,x_n) \in  \C[x_1,\dots,x_n]$ and $0<y$ small enough to satisfy equation (\ref{multi}),
\[0 \leq  (-1)^\ell \int_{S_n(P,y)} \log^\ell|P(x_1,\dots,x_n)| \frac{dx_1 }{x_1}\cdots \frac{dx_n }{x_n} \leq  C_P J_{\ell,\delta_P}(y).\]
\end{lem}
\begin{proof}
The case $\ell=1$ is Lemma 4 in \cite{L}. The general proof starts in the same way. 
Define for $0<z\leq 1$
\[h(z):=\mu_1(S(P,z)),\]
where we recall that $\mu_1$ stands for the Lebesgue measure of the set.  
Let the leading coefficient of $P(x)$ be $a$ with $|a|\geq 1$. Then $a^{-1}P$ is monic and so Theorem \ref{thm1L} implies that
\[h(z)\leq C_k\left(\frac{z}{|a|}\right)^\frac{1}{k-1} \leq C_k z^\frac{1}{k-1}.\]
Now we compute the desired integral.
\begin{eqnarray*}(-1)^\ell \int_{S(P,y)} \log^\ell|P(x)| \frac{dx}{x} &=&(-1)^\ell\int_{z=0}^{z=y} \int_{{|x|=1}\atop{|P(x)|=z}} \log^\ell z \frac{dx}{x} dz\\ &=&(-1)^\ell\int_{0}^{y} \log^\ell z h'(z) dz\\
&=& (-1)^\ell \log^\ell y h(y)  - \int_{0}^{y} \frac{d}{dz}\left[(-\log z)^\ell\right] h(z) dz\\
& \leq& (-1)^\ell \log^\ell y C_k y^\frac{1}{k-1}  - \int_{0}^{y} \frac{d}{dz}\left[(-\log z)^\ell\right] C_k z^\frac{1}{k-1} dz
\end{eqnarray*}
where the last inequality is consequence of the fact that $(-\log z)^\ell$ is a positive decreasing function and its derivative is negative. By applying integration by parts again we obtain
\[\leq (-1)^\ell C_k \int_0^y \log^\ell z d \left(z^\frac{1}{k-1} \right),\]
which finishes the proof of the first statement by Lemma \ref{integral} and equation (\ref{I})

The proof of the second statement follows along the same lines.

\end{proof}

\begin{lem}\label{intersection}
 Let $P_1(x),\dots, P_s(x) \in \C[x]$  be polynomials having $k_1,\dots,k_s$ nonzero complex coefficients with absolute value greater than
 1 and $0< y_1,\dots, y_s\leq1$. Let $1\leq n\leq s$. Then  
\[0\leq (-1)^{s} \int_{\bigcap_{i=1}^n S(P_i,y_i) \setminus \bigcup_{i=n+1}^s  S(P_i,y_i)} \log|P_1(x)| \cdots \log|P_s(x)| \frac{dx}{x}\]
\[\leq \left( C_{k_1} I_{n, k_1}(y_1)\cdots  C_{k_n} I_{n, k_n}(y_n) \right)^\frac{1}{n} (-1)^{s-n} \log y_{n+1}\cdots \log y_s.\]
\end{lem}
\begin{proof}
Notice that  $0\leq -\log |P(x)|\leq -\log y$ for $x\not \in S(P, y)$ for $0<y\leq 1$. Therefore,
\begin{eqnarray*}
&&(-1)^{s} \int_{\bigcap_{i=1}^n S(P_i,y_i) \setminus \bigcup_{i=n+1}^s  S(P_i,y_i)} \log |P_1(x)| \cdots \log |P_s(x)| \frac{dx}{x}\\
&&\leq (-1)^{s} \log y_{n+1}\cdots \log y_s\int_{\bigcap_{i=1}^n S(P_i,y_i) \setminus \bigcup_{i=n+1}^s  S(P_i,y_i)} \log |P_1(x)| \cdots \log |P_n(x)| \frac{dx}{x} \\
&&\leq (-1)^{s} \log y_{n+1}\cdots \log y_s\int_{\bigcap_{i=1}^n S(P_i,y_i)} \log |P_1(x)| \cdots \log |P_n(x)| \frac{dx}{x} \\
&&\leq (-1)^{s-n} \log y_{n+1}\cdots \log y_s \left( C_{k_1} I_{n,k_1}(y_1)\cdots  C_{k_n} I_{n,k_n}(y_n) \right)^\frac{1}{n}\\
\end{eqnarray*}
by Lemma \ref{superHolder} and Lemma \ref{Lemma}.
\end{proof}
\begin{lem}\label{cup}
Let $P_1(x),\dots, P_s(x) \in \C[x]$ be polynomials having $k_1,\dots,k_s$ nonzero complex coefficients with absolute value greater than
1 and $0< y_1,\dots, y_s\leq 0$.  
Then
\[0\leq (-1)^{s} \int_{S(P_1,y_1)\cup \dots \cup S(P_s,y_s)} \log |P_1(x)| \cdots \log |P_s(x)| \frac{dx}{x}\]
\[\leq \sum_{A\subset\{1,\dots,s\}} \prod_{i \in A} \left(C_{k_i} I_{|A|,k_i}(y_i)\right)^\frac{1}{|A|} \prod_{i \in \{1,\dots,s\} \setminus A} (-\log y_{i}).\]
\end{lem}
\begin{proof}
We start with the observation that
\[\bigcup_{i=1}^s S(P_i,y_i)=\bigcup_{A\subset\{1,\dots,s\}} \left(\bigcap_{i\in A} S(P_i,y_i) \setminus \bigcup_{i \in \{1,\dots, s\}\setminus A}   S(P_i,y_i) \right).\]
By applying Lemma \ref{intersection}, we get 
\begin{eqnarray*}
&&(-1)^{s} \int_{S(P_1,y_1)\cup \dots \cup S(P_s,y_s)} \log |P_1(x)| \cdots \log |P_s(x)| \frac{dx}{x}\\
&&\leq \sum_{A\subset\{1,\dots,s\}}(-1)^{s} \int_{\bigcap_{i\in A} S(P_i,y_i) \setminus \bigcup_{i \in \{1,\dots, s\}\setminus A}   S(P_i,y_i) } \log |P_1(x)| \cdots \log |P_s(x)| \frac{dx}{x}\\
&&\leq \sum_{A\subset\{1,\dots,s\}} \prod_{i \in A} \left(C_{k_i} I_{|A|,k_i}(y_i)\right)^\frac{1}{|A|} \prod_{i \in \{1,\dots,s\} \setminus A} (-\log y_{i}).
 \end{eqnarray*}

\end{proof}

Setting $y_1=\dots=y_s=y$ and letting $y \rightarrow 0$, we get the following result by Corollary \ref{limI}.

\begin{cor}\label{cuplimit}Let $P_1(x),\dots, P_s(x) \in \C[x]$ be polynomials having $k_1,\dots,k_s$ nonzero complex coefficients.
% with absolute value greater than1 . 
Let $0<y<1$. As $y$ approaches 0, we obtain
\[\lim_{y \rightarrow 0} \int_{S(P_1,y)\cup \dots \cup S(P_s,y)} \log |P_1(x)| \cdots \log |P_s(x)| \frac{dx}{x}=0,\]
where the speed of convergence is independent of the polynomials $P_1(x),\dots,P_s(x)$. 
\end{cor}

\begin{lem}\label{partintersection}
 Let $P_1(x),\dots, P_s(x) \in \C[x]$ be polynomials having $k_1,\dots,k_s$ nonzero complex coefficients with absolute value greater than
 1 and $0< y_1,\dots, y_s\leq 1$.  Then  
\begin{eqnarray*}
0&\leq& (-1)^{s} \int_{S(P_1,y_1)\cap \dots \cap S(P_s,y_s)} \log|P_1(x)| \cdots \log|P_s(x)| \frac{dx}{x}\\
&\leq& \left( C_{k_1} I_{s,k_1}(y_1)\cdots  C_{k_s} I_{s,k_s}(y_s) \right)^\frac{1}{s}.
\end{eqnarray*} 
\end{lem}
\begin{proof}
 This is a consequence of Lemma \ref{intersection} with $n=s$. 
\end{proof}

\begin{lem}
 Let $P_1(x),\dots, P_s(x) \in \C[x]$ be polynomials having $k_1,\dots,k_s$ nonzero complex coefficients with absolute value greater than
 1 and $0< y_1,\dots, y_s\leq 1$.  Then  
\begin{eqnarray*}
0&\leq& \int_{S(P_1,y_1)\cap \dots \cap S(P_s,y_s)} \max_{1\leq i \leq s}\{\log |P_i(x)|\} \frac{dx}{x}\\
&\leq& (2\pi)^{1-\frac{1}{s}}\left( C_{k_1} I_{s,k_1}(y_1)\cdots  C_{k_s} I_{s,k_s}(y_s) \right)^\frac{1}{s}.
\end{eqnarray*} 
\end{lem}

\begin{proof}
Notice that $\max_{1\leq i \leq s}\{\log |P_i(x)|\}= -\min_{1\leq i \leq s}\{-\log |P_i(x)|\}$. In $S(P_1,y_1)\cap \dots \cap S(P_s,y_s)$, we have
$0\leq \min_{1\leq i \leq s}\{-\log |P_i(x)|\}\leq -\log|P_i(x)|$ for any $i=1, \dots, s$. Thus,
\[
 \left(-\max_{1\leq i \leq s}\{\log |P_i(x)|\}\right)^s=\left(\min_{1\leq i \leq s}\{-\log |P_i(x)|\}\right)^s \leq (-1)^s \log|P_1(x)| \cdots \log|P_s(x)|.
\]
 By applying H\"older's inequality, and taking into account that the measure of $S(P_1,y_1)\cap \dots \cap S(P_s,y_s)$ is bounded by $2 \pi$, we get
\begin{eqnarray*}
0 &\leq& \int_{S(P_1,y_1)\cap \dots \cap S(P_s,y_s)} -\max_{1\leq i \leq s}\{\log |P_i(x)|\} \frac{dx}{x}\\
&\leq &(2\pi)^{1-\frac{1}{s}} \left(\int_{S(P_1,y_1)\cap \dots \cap S(P_s,y_s)} \left(-\max_{1\leq i \leq s}\{\log |P_i(x)|\}\right)^s \frac{dx}{x}\right)^\frac{1}{s}\\
&\leq & (2\pi)^{1-\frac{1}{s}} \left( C_{k_1} I_{s,k_1}(y_1)\cdots  C_{k_s} I_{s,k_s}(y_s) \right)^\frac{1}{s^2}.
\end{eqnarray*}
\end{proof}

Again, we let $y_1=\dots=y_s=y$ and $y \rightarrow 0$ and we conclude the following result. 
\begin{cor}\label{caplimit}Let $P_1(x),\dots, P_s(x) \in \C[x]$ be polynomials having $k_1,\dots,k_s$ nonzero complex coefficients. Let $0<y \leq 1$. As $y$ approaches 0, we obtain
\[\lim_{y \rightarrow 0} \int_{S(P_1,y)\cap \dots \cap S(P_s,y)} \max_{1\leq i \leq s}\{\log |P_i(x)|\} \frac{dx}{x}=0,\]
where the speed of convergence is independent of the polynomials $P_1(x),\dots,P_s(x)$. 
\end{cor}
Observe that when $k_i=1$, the previous result is trivially true since the set $S(P_i,y)$ becomes empty for $y$ sufficiently small. 

\begin{remark}\label{superrem} Results analogous to Corollary \ref{cuplimit} and Corollary \ref{caplimit}  can be proved for the case where $P_1(x_1,\dots,x_n),\dots,P_s(x_1,\dots,x_n)$ are {\em fixed} polynomials in $\C[x_1,\dots,x_n]$.   
\end{remark}

\section{Proof of Theorem \ref{superBL}}
We begin by first proving that the extended versions of Mahler measures always exist (i.e., that the integrals always converge).
 This was used repeatedly in previous works but the details have never been written and we include them here for completeness. 
\begin{thrm}\label{existence} Let $P_1(x_1,\dots,x_n),\dots,P_s(x_1,\dots,x_n) \in \C[x_1,\dots,x_n]$ nonzero polynomials. Then the integrals giving the generalized Mahler measure and the multiple Mahler measure converge, i.e.,
\begin{enumerate}
\item \[|\m_{\max}(P_1,\dots,P_s)|<\infty,\]
\item  \[|\m({P_1},\dots, {P_s})|<\infty.\]
\item If $P_1=\cdots =P_s=P$, \[ \m_s(P)<\infty.\] 
 \end{enumerate}
\end{thrm}
\begin{proof}
(1) Let $y>0$. We write
\begin{eqnarray*}&&\int_{\TT^n}\max_{1\leq i \leq s}\{\log|P_i(x_1,\dots,x_n)|\}\frac{dx_1}{x_1}\cdots \frac{dx_n}{x_n} \\&&= \int_{{S_n(P_1,y)\cap \dots \cap S_n(P_s,y)}}\max_{1\leq i \leq s}\{\log|P_i(x_1,\dots,x_n)|\} \frac{dx_1}{x_1}\cdots \frac{dx_n}{x_n}\\
&&+ \int_{{S(P_1,y)^c\cup \dots \cup S(P_s,y)^c}}\max_{1\leq i \leq s}\{\log|P_i(x_1,\dots,x_n)|\}\frac{dx_1}{x_1}\cdots \frac{dx_n}{x_n}. 
\end{eqnarray*}
The second integral converges, while the first integral approaches 0 as $y \rightarrow 0$ by Corollary \ref{caplimit} and Remark \ref{superrem}. Therefore, the integral on the left converges. 

(2)  For $y>0$, we consider
\begin{eqnarray*}
&&\int_{\TT^n} \log |P_1(x_1,\dots,x_n)| \dots \log |P_s(x_1,\dots,x_n)| \frac{dx_1}{x_1}\cdots \frac{dx_n}{x_n}\\ &&= \int_{S_n(P_1,y)\cup \dots \cup S_n(P_s,y)} \log |P_1(x_1,\dots,x_n)| \dots \log |P_s(x_1,\dots,x_n)| \frac{dx_1}{x_1}\cdots \frac{dx_n}{x_n}\\
&&+\int_{S_n(P_1,y)^c\cap \dots \cap S_n(P_s,y)^c} \log |P_1(x_1,\dots,x_n)| \dots \log |P_s(x_1,\dots,x_n)| \frac{dx_1}{x_1}\cdots \frac{dx_n}{x_n}.
\end{eqnarray*}
As before, the second integral converges, while the first integral approaches 0 as $y \rightarrow 0$ by the Corollary \ref{cuplimit} and Remark \ref{superrem}. Thus, the first integral converges.

(3) This statement is a particular case of (2). 
\end{proof}

\begin{proof}[Theorem \ref{superBL}]
(1) Following \cite{L}, we define $F: \TT^n \rightarrow \R$ by $F(\omega)=-\max_{1\leq i \leq s}\{\log |P_i(\omega)|\}$
for $\omega \in \TT^n$. It suffices to prove that 
\[\lim_{q(\bf{r})\rightarrow \infty} \left| \int_\TT F_{\bf{r}} -\int_{\TT^n} F \right| = 0.\] 
Without loss of generality, we may assume that each coefficient of $P_i$ has modulus $\geq 1$, 
and therefore the same is true for $P_{i,\bf{r}}$ for $q(\bf{r})$ sufficiently large. For any $0\leq  y \leq 1$ we construct a continuous function
$g_{y} : \TT^n \rightarrow \R$ such that $0\leq g_{y}(\omega) \leq 1$ for all $\omega \in \TT^n$, $g_{y}(\omega)=1$ for $\max_{1\leq i \leq s}\{|P_i(\omega)|\}\geq y$, and 
$g_{y}(\omega)=0$ for $\max_{1\leq i \leq s}\{|P_i(\omega)|\}\leq \frac{1}{2}y$. Therefore, $g_y F_{\bf{r}}$ is a continuous function on $\TT^n$ for $0\leq  y \leq 1$. Since
$F=g_yF+(1-g_y)F$, the triangle inequality implies that  

\begin{align} \label{sumgen}
& \limsup_{q(\bf{r})\rightarrow \infty} \left| \int_\TT F_{\bf{r}} -\int_{\TT^n} F\right| 
\leq \limsup_{q(\bf{r})\rightarrow \infty}  \left| \int_\TT [g_y F]_{\bf{r}} 
-\int_{\TT^n} g_y F\right|  \\
\nonumber &+\limsup_{q(\bf{r})\rightarrow \infty}  \left| \int_\TT \left[\left(1-g_{y}\right) F \right]_{\bf{r}}\right|+ \limsup_{q(\bf{r})\rightarrow \infty}  \left| \int_{\TT^n} \left(1-g_{y}\right) F \right| 
\end{align}
Now, by the Weierstrass approximation theorem, the first term goes to zero since $g_{y}F$  is continuous on $\TT^n$.
The function $[(1-g_y) F]_{\bf{r}} = (1-g_{y,{\bf{r}}}) F_{\bf{r}}$ vanishes
 in the set $\bigcup S(P_{i,\bf{r}},y)^c=\left(\bigcap S(P_{i,\bf{r}},y)\right)^c$ and it is bounded below by 0 and above by
$F_{\bf{r}} $ in $\bigcap S(P_{i,\bf{r}},y)$. This implies
\[0\leq \limsup_{q(\bf{r})\rightarrow \infty}  \left| \int_\TT \left[\left(1-g_{y}\right)F \right]_{\bf{r}}\right|\leq \limsup_{q(\bf{r})\rightarrow \infty}  \left| \int_{\bigcap S(P_{i,\bf{r}},y)} F_{{\bf{r}}}\right|,\]
which goes to zero as $y\rightarrow 0$ by Corollary \ref{caplimit}.
Finally, the third term in \eqref{sumgen} tends to 0  as $y \rightarrow 0$ since $F$ is integrable 
over $\TT^n$ by Theorem \ref{existence} (1).

Thus, $\limsup_{q(\bf{r})\rightarrow \infty} \left| \int_\TT F_{\bf{r}} -\int_{\TT^n} F\right|=0$ since it is independent of $y$ and tends to zero as $y \rightarrow 0$. 

(2) We proceed as before. We define $F: \TT^n \rightarrow \R$ by $F(\omega)=\prod_{i=1}^s(-\log |P_i(\omega)|)$
for $\omega \in \TT^n$. %It suffices to prove that 
%\[\lim_{q(\bf{r})\rightarrow \infty} \left| \int_\TT F_{\bf{r}} -\int_{\TT^n} F\right| = 0.\] 
Without loss of generality, we may assume that each coefficient of $P_i$ has modulus $\geq 1$, 
and therefore the same is true for $P_{i,\bf{r}}$ for $q(\bf{r})$ sufficiently large. For any $0\leq y \leq 1$ we construct a continuous function
$g_{y} : \TT^n \rightarrow \R$ such that $0\leq g_{y}(\omega) \leq 1$ for all $\omega \in \TT^n$, $g_{y}(\omega)=1$ if $|P_i(\omega)|\geq y$ for all $i$, and 
$g_{y}(\omega)=0$ if there is an $i$ such that $|P_i(\omega)|\leq \frac{1}{2}y$. Therefore, $g_{y}F$ is a continuous function on $\TT^n$ for $0\leq y \leq 1$. The triangle inequality implies that  

\begin{align}\label{sumhigh}
& \limsup_{q(\bf{r})\rightarrow \infty}  \left| \int_\TT F_{\bf{r}} -\int_{\TT^n} F\right| \leq \limsup_{q(\bf{r})\rightarrow \infty}  \left| \int_\TT [g_{y}F]_{\bf{r}} -\int_{\TT^n} g_{y}F\right|  \\
\nonumber &+\limsup_{q(\bf{r})\rightarrow \infty}  \left| \int_\TT \left[\left(1-g_{y}\right)F\right]_{\bf{r}}\right|+ \limsup_{q(\bf{r})\rightarrow \infty}  
\left| \int_{\TT^n} \left(1-g_{y}\right)F\right| 
\end{align}
The Weierstrass approximation theorem implies that the first term goes to zero since $g_{y}F$  is continuous on $\TT^n$.
Now the function $\left[\left(1-g_{y}\right)F\right]_{\bf{r}}=
\left(1-g_{y,{\bf{r}}}\right)F_{{\bf{r}}}$ vanishes in the set $\bigcap S(P_{i,\bf{r}},y)^c=\left(\bigcup S(P_{i,\bf{r}},y)\right)^c$ and it is bounded below by 0 and above by $F_{{\bf{r}}}$ in $\bigcup S(P_{i,\bf{r}},y)$. Combining all of this, 
\[0\leq \limsup_{q(\bf{r})\rightarrow \infty}  \left| \int_\TT \left[\left(1-g_{y}\right)F\right]_{\bf{r}}\right| \leq \limsup_{q(\bf{r})\rightarrow \infty}  \left| \int_{\bigcup S(P_{i,\bf{r}},y)} F_{\bf{r}}\right|.\]
The term in the right goes to zero as $y\rightarrow 0$ by Corollary \ref{cuplimit}. The third term in \eqref{sumhigh} tends to 0 as $y \rightarrow 0$ since $F$ is integrable 
over $\TT^n$ by Theorem \ref{existence} (2). 

Finally, $\limsup_{q(\bf{r})\rightarrow \infty} \left| \int_\TT F_{\bf{r}} -\int_{\TT^n} F\right| =0$ since it is independent of $y$ and tends to zero as $y \rightarrow 0$. 

(3) This case follows from (2) by setting $P_1=\dots=P_s=P$. This concludes the proof of the theorem.

\end{proof}

\proof[Acknowledgement] The authors are thankful to David Boyd for his encouragement and to the referee for helpful corrections.

\end{document}